\documentclass[11pt]{amsart}
\usepackage{amsthm}
\usepackage{mathrsfs}
\usepackage{amssymb}
\usepackage{amsmath}
\usepackage{hyperref}
\usepackage{xcolor}
\usepackage{amsthm}
\usepackage{dsfont}
\usepackage[matrix,arrow]{xy}

\usepackage{enumerate}

\usepackage[style=trad-abbrv,maxnames=99,maxalphanames=9, isbn=false, giveninits=true, doi=false, url=true]{biblatex}
\bibliography{Hermitianref.bib}
\renewbibmacro{in:}{}

\DeclareMathAlphabet\mathbfcal{OMS}{cmsy}{b}{n}

\numberwithin{equation}{section}

\newtheorem{theorem}{Theorem}[section]

\newtheorem{lemma}[theorem]{Lemma}

\newtheorem{proposition}[theorem]{Proposition}
\newtheorem{corollary}[theorem]{Corollary}
\newtheorem{remark}[theorem]{Remark}

\newtheoremstyle{named}{}{}{\itshape}{}{\bfseries}{.}{.5em}{\thmnote{#3\ }#1}
\theoremstyle{named}

 \newcommand{\RR}{\mathbb{R}}
 \newcommand{\CC}{{\mathbb C}}

\newcommand{\ddc}{\sqrt{-1}\partial\bar\partial}

\newcommand{\exph}{\mathrm{exph}}
\newcommand{\dist}{\mathrm{dist}}
\def\Xint#1{\mathchoice
{\XXint\displaystyle\textstyle{#1}}%
{\XXint\textstyle\scriptstyle{#1}}%
{\XXint\scriptstyle\scriptscriptstyle{#1}}%
{\XXint\scriptscriptstyle\scriptscriptstyle{#1}}%
\!\int}
\def\XXint#1#2#3{{\setbox0=\hbox{$#1{#2#3}{\int}$ }
\vcenter{\hbox{$#2#3$ }}\kern-.6\wd0}}

\def\dashint{\Xint-}

\title{Modulus of continuity for solutions to complex Monge-Amp\`ere equations on Hermitian manifolds
}

\usepackage{hyperref}
\hypersetup{colorlinks,linkcolor={blue},citecolor={red}}

\begin{document}

\author{Junbang Liu}
\address{Department of Mathematics, Hong Kong University of Science and Technology, Clear Water Bay.}
\email{junbangliu@ust.hk}
\begin{abstract}
    In this note, we give a proof of the uniform log-continuity of the solution to complex Monge-Amp\`ere equations on compact Hermitian manifolds, which is a generalization of the result of \cite{GPTW} in the K\"ahler case.
\end{abstract}
\maketitle

\section{introduction}
Let $(M^n,\omega_M)$ be a $n$-dim compact Hermitian manifold with a fixed Hermitian metric $\omega_M$. Let $F\in C^\infty(M,\RR)$ be a smooth real function. Then it's proved by Tosatti and Weinkove in \cite{TW10} that the following volume conjecture holds: there exists an unique $\varphi\in C^\infty(M,\RR)$ and unique constant $c$ such that \begin{equation}\label{maequation}
    (\omega_M+\ddc u)^n=e^{F+c}\omega_M^n,\quad \omega_M+\ddc u>0,\quad \sup_M u=0.
\end{equation}

The theorem we proved in this note is \begin{theorem}\label{mainthm}
    Given $p>n$, for any $0<\alpha<\frac{p}{n}-1$, there exists a constant $C$ depending on $M,\omega_M,p,\alpha$, and upper bound of $\int_M(|F|^p+1)e^F\omega_M^n$, positive lower bound of $\int_Me^{\frac{1}{n}F}\omega_M^n$, such that solutions to \eqref{maequation} satisfies \[
    |u(x)-u(y)|\leq \frac{C}{|\log dist_{\omega_M}(x,y)|^{\alpha}}.
    \]
\end{theorem}
It's a generalization of the theorem in \cite{GPTW,GG23} to the Hermitian manifolds. 
We remark that in the K\"ahler case, the exponent $\alpha$ in the $\log$-continuity can be $\frac{p}{n}-1$. The H\"older regularity of the solutions to complex Monge-Amp\`ere equation on compact Hermitian manifolds have been studied by many authors, see \cite{MR4398254,KN19} and the reference therein, after the internsively developed pluripotential theorey on Hermitian manifolds(see \cite{Di16,Ng16,KN15}). Here we use a purely PDE method developed by Guo-Phong-Tong in \cite{GP23}.  

In the following, we say a constant is uniform if it depends on the parameters stated in the above theorem. 

\section{stability estimate}
The following stability estimate is crucial to our proof. 
\begin{proposition}\label{stabilityestimate}
    For any $v\in Psh(\omega_M)\cap C^\infty(M,\RR^+)$, there exists a uniform constant $C$ additionally depending on $||v||_{L^\infty}$ such that \[
    \sup(v-u)\leq \frac{C}{ |\log(||(v-u)_+||_{L^1})|^{\alpha}}.
    \]

\end{proposition}

To begin with, we need a lemma to estimate the constant $c$ under the assumption of a weaker bound on the volume density. 
\begin{lemma}\label{estimatesupslope}
    Let $(u,c)$ be a smooth solution to equation \eqref{maequation} and $\Phi:\RR\rightarrow \RR^+$ be a positive increasing function with $\int_1^\infty\Phi^{-\frac{1}{n}}(t)dt<\infty$.  Then there exists a constant $C_0$ depending on $M,\omega_M, \Phi$ and upper bounds for $\int_M\Phi(F)e^F\omega_M^n,$ and positive lower bound of $\int_Me^{\frac{F}{n}}\omega_M^n$, such that \[e^c+e^{-c}+||u||_{L^\infty}\leq C_0.\] 
\end{lemma}
\begin{proof}
The lower bound $c$ is proved in \cite{liu2024complexalexandrovbakelmanpucciestimateapplications} by the complex ABP estimate. The upper bound of $c$ is proved in \cite{KN15}  lemma 5.9, and it depends on the positive lower bound of $\int_Me^{\frac{F}{n}}\omega_M^n, M,\omega_M$. We only need to prove the $L^\infty$-estimate. Here, we introduce another method to construct a good comparison metric.  We cover the manifold $(M,\omega_M)$ by coordinates balls $\{B(x_i,2r_i)|i=1,\cdots,n\}$ such that
$B_i:=B(x_i,r_i)$ still covers $M$. Let $\rho_i$ be smooth functions on $M$, such that $2B_i=\{\rho_i\leq 0\}$ and $ \rho_i=|z|^2-4r_i^2$ on $2B_i$. On each $2B_i$, by the  work of Caffarelli-Kohn-Nirenberg-Spruck \cite{CNS2}, one can solve the following equation \begin{equation}
    (\ddc\psi_i)^n=\Phi(F)e^{F}\omega_M^n, \quad \ddc\psi_i>0, \quad \psi_i|_{\partial 2B_i}=-1.
\end{equation}
By the complex ABP estimate in \cite{liu2024complexalexandrovbakelmanpucciestimateapplications}, $||\psi_i||_{L^\infty}$ is uniformly bounded by a constant $C_1$. Then by choosing $C_2$ large enough with $3C_2r_i^2>C_1+1$ and $\epsilon$ small, let $\tilde{\psi}_i:=\max_{\epsilon}\{C_2\rho_i, \psi_i\}$, where $\max_\epsilon$ is the regularized max function, c.f. lemma 5.18 in Demailly's book \cite{DemBook}. Then the function $\tilde{\psi}_i$ satisfies: \begin{itemize}
    \item $\tilde{\psi}_i=\psi_i$ on $B_i$(since $C_2\rho_i<-C_1-1<\psi_i$ on $B_i$ with our choice of $C_2$);
    \item $\ddc\tilde{\psi}_i\geq -C_3\omega_M$ on $M$, for some uniform constant $C_3$ depending on $M,\omega_M$(since $\rho_i$ is smooth on $M$, and $\tilde{\psi}_i=\rho_i$ in a neighborhood of $M\setminus 2B_i$). 
\end{itemize} 
Let $\psi:=\frac{1}{C_3N}\sum_{i=1}^{N}\tilde{\psi}_i$.
Therefore, one has on each $B_i$, \[(\omega_M+\ddc\psi)^n\geq \frac{1}{(NC_3)^n}(\ddc\psi_i)^n\geq \epsilon_0\Phi(F)e^F\omega_M^n \] for some uniform constant $\epsilon_0$. Once one has such a comparison metric, the proof of $L^\infty$-estimate follows from \cite{liu24c}.
\begin{remark}It's interesting to know whether the constant c can be estimated from below when we remove the condition $\int_1^\infty\Phi^{-\frac{1}{n}}<\infty$? If so, it can be applied to get Trudinger type estimates for solutions to Monge-Amp\`ere equations on Hermitian manifolds like in the K\"ahler case.
\end{remark}
\end{proof}
With this lemma, we can construct appropriate comparison Hermitian metrics to reduce the proof to $L^p$-case. 
Let $\Phi(x)=x^{n\alpha}$ for $x>1$, and extend it smoothly to be a positive increasing function on $\RR$. Then \[
    \int_{M}(|\log\Phi(F)+F|^{p-n\alpha}+1)\Phi(F)e^F\omega_M^n\leq C\int_M(|F|+1)^pe^{F}\omega_M^n.
\]  
Let $\psi_1$ be the solution to equation \begin{equation}
    (\omega_M+\ddc\psi_1)^n=e^{c_1}\Phi(F)e^F\omega_M^n,\quad \omega_M+\ddc\psi_1>0, \quad \sup \psi_1=0.
\end{equation}
Then by lemma \ref{estimatesupslope}, the constant $e^{-c_1}$ and $||\psi_1||_{L^\infty}$ is uniformly bounded from above by a uniform constant $C_0$. Let $\alpha_M$ be Tian's $\alpha$-invariant of $(M,\omega_M)$.
For any $q,\Lambda_1>0, a\in (0,1)$, set  $h(s)=-\frac{q}{a\alpha_M}e^{\frac{c-c_1}{n}}\int_s^{C_0+\Lambda_1}\Phi^{-\frac{1}{n}}(t)dt$, and $\psi_2=-h(-\frac{\alpha_M}{q}\psi_1+\Lambda_1)$. Without loss of generality, we can assume that $\sup(-\frac{\alpha_M}{q}\psi)\leq C_0$, so in this case, we get that $\psi_2$ is a nonnegative function.  We have \begin{lemma}\label{comparisonlemma}
    The inequality holds\begin{equation}
        \omega_u^n\leq a^n\omega_{\psi_2}^n+f\omega_M^n,
    \end{equation}with $f=\exp{(-\frac{\alpha}{q}\psi_1+\Lambda_1+c)}$ and $\Lambda_1$ is chosen so that $\Phi(\Lambda_1)^{\frac{1}{n}}=\frac{1}{a}e^{\frac{c-c_1}{n}}.$
\end{lemma}

\begin{proof}
    From our choice of $\Lambda_1$, we have $\frac{\alpha}{q}h'(\Lambda_1)=1$. Note that $h$ is an increasing concave function, we have \[\begin{aligned}
            \omega_M+\ddc\psi_2=&\omega_M+\frac{\alpha}{q}h'\ddc\psi_1-\frac{\alpha^2}{q^2}h''\sqrt{-1}\partial\psi_1\wedge\bar{\partial}\psi_1\geq \frac{\alpha}{q}h'(\omega_M+\ddc\psi_1)
    \end{aligned}
    \]In the above inequality we have used that $\frac{\alpha}{q}h'(-\frac{\alpha}{q}\psi_1+\Lambda_1)\leq \frac{\alpha}{q}h'(\Lambda_1)=1.$
    Hence $$\omega_u^n=e^{F+c}\omega_M^n\leq \frac{q^n}{\alpha^n}\frac{1}{h'^n}\frac{e^c}{e^{c_1}}\frac{1}{\Phi(F)}\omega_{\psi_2}^n.$$  Next, for any $a\in(0,1)$,
    if $\frac{q^n}{\alpha^n}\frac{1}{h'^n}\frac{e^c}{e^{c_1}}\frac{1}{\Phi(F)}\leq a^n$, we get \[\omega_u^n\leq a^n\omega_{\psi_2}.\] Otherwise, we have $\frac{q^n}{\alpha^n}\frac{1}{h'^n}\frac{e^c}{e^{c_1}}\frac{1}{\Phi(F)}>a^n$, which is equivalent to $\Phi(F)<\frac{q^n}{a^n\alpha^n}\frac{1}{h'^n(-\alpha q^{-1}\psi_1+\Lambda_1)}e^{c-c_1}=\Phi(-\frac{\alpha}{q}\psi_1+\Lambda_1),$ where we use the definition of $h$ in the last equality. From the monotonicity of $\Phi$, we get $F\leq -\frac{\alpha}{q}\psi_1+\Lambda_1$. So, in this case, we have $e^{F+c}\leq f$ with $f$ stated in the lemma. 
\end{proof}

For $a\in(0,1)$ and ${p_1}>1$,  which will be specified later, let $\tilde{u}:=-u+(1-2a)v+a\psi_2$, 
and $
\Omega_s:=\{\tilde{u}-s>0
\}.$ We use a trick of Cheng-Xu in \cite{CX24}, and consider the auxiliary equation \[
(\omega_M+\ddc\psi_3)^n=\frac{e^{c_{s,j,{p_1}}}\eta_{j}(\tilde{u}-s)f\omega_M^n}{A_{s,j,{p_1}}},\quad  \omega_M+\ddc\psi_3>0,\quad \sup\psi_3=0,
\]where $A_{s,j,{p_1}}=\left(\int_M\eta_{j}^{{p_1}}(\tilde{u}-s)f^{{p_1}}\omega_M^n\right)^{\frac{1}{{p_1}}},$ and $\eta_j(x)$ is a sequence of positive smooth function approximating $\max\{x,0\}$ from above in $L^\infty$ sense. From dominated convergence theorem, it's easy to see \[
\lim_{j\rightarrow \infty}A_{s,j,{p_1}}=\left(\int_{\Omega_s}(\tilde{u}-s)_+^{{p_1}} f^{{p_1}}\omega_M^n\right)^\frac{1}{{p_1}}=:A_{s,{p_1}}.
\] 
Then apply lemma \ref{estimatesupslope} again with $\Phi(x)=e^{({p_1}-1)x}$, the constant $c_{s,j,{p_1}}$ and solution $\psi_{3}$ are uniformly bounded. 
Choose constant $\epsilon_3,\Lambda_3$ suitably, then we have \begin{lemma}
    The inequality holds: 
    
    $$\tilde{u}-s\leq \epsilon_3(\psi_3+\Lambda_3)^{\frac{n}{n+1}}$$ for \[
    \epsilon_3=\left(\frac{n+1}{n}\right)^{\frac{n}{n+1}}e^{-c_{s,j,{p_1}}}A_{s,j,{p_1}}^\frac{1}{n+1}, \quad \Lambda_3=\frac{n}{n+1}\frac{A_{s,j,{p_1}}}{a^{n+1}}e^{-(n+1)c_{s,j,{p_1}}}.
    \]
\end{lemma}
\begin{proof}
    Without loss of generality, we can assume the maximum of $\tilde{u}-s-\epsilon_3(-\psi_{3}+\Lambda_3)^{\frac{n}{n+1}}$ is attained at some point $x_0\in \Omega_s$; otherwise the inequality holds automatically. Then at $x_0$, by maximum principle, we have \begin{equation}
        \begin{aligned}
            0\geq& \ddc\tilde{u}+\epsilon_3\frac{n}{n+1}(-\psi_3+\Lambda_3)^{-\frac{1}{n+1}}\ddc\psi_3+\epsilon\frac{n^2}{n+1}\sqrt{-1}\partial\psi_3\wedge\bar{\partial}\psi_3\\
            \geq&-\omega_u+(1-2a)\omega_v+a\omega_{\psi_2}+\epsilon_3\frac{n}{n+1}(-\psi_3+\Lambda_3)^{-\frac{1}{n+1}}\omega_{\psi_3}+(a-\epsilon_3\frac{n}{n+1}(\Lambda_3)^{-\frac{1}{n+1}})\omega_M
        \end{aligned}
    \end{equation}
    From our choice of $\Lambda_3$, the last term above is nonnegative. So we have \begin{equation}\begin{aligned}\label{auximaxprin}
                \omega_u^n\geq &a^n\omega_\psi^n+\left(\epsilon_3\frac{n}{n+1}\right)^n(-\psi_3+\Lambda_3)^{-\frac{n}{n+1}}\omega_{\psi_3}^n\\
                =&a^n\omega_\psi^n+\left(\epsilon_3\frac{n}{n+1}\right)^n\frac{\eta_j(\tilde{u}-s)}{(-\psi_3+\Lambda_3)^{\frac{n}{n+1}}}\frac{e^{c_{s,j,{p_1}}}f\omega_M^n}{A_{s,j,{p_1}}}
    \end{aligned}
    \end{equation}
    Then apply lemma \ref{comparisonlemma}  to the left-hand side of \eqref{auximaxprin}, we get \[
    \eta_{j}(\tilde{u}-s)\leq \left(\epsilon_3\frac{n}{n+1}\right)^{-n}e^{-c_{s,j,{p_1}}}A_{s,j,{p_1}}(-\psi_3+\Lambda_3)^{\frac{n}{n+1}}.
    \]Note that $\eta_j(x)\geq \max\{x,0\}$ and from our choice of $\epsilon_3$, we have $\epsilon_3=\left(\epsilon_3\frac{n}{n+1}\right)^{-n}e^{-c_{s,j,{p_1}}}A_{s,j,{p_1}}$. The inequality in the lemma follows.
\end{proof}

Let $\phi(s)=\int_{\Omega_s}f\omega_M^n$, one can estimate $A_{s,{p_1}}$ as 
\begin{lemma}
    If $\frac{A_{s,{p_1}}}{a^{n+1}}\leq 1$, and fixed $p_1>1$ close to $1$, one has $A_{s,{p_1}}\leq Ce^{A_1\Lambda_1}\phi(s)^{1+\delta_n}$ for some uniform constant $C$, and dimentional constant  $A_1, \delta_n>0$.
\end{lemma}

\begin{proof}
When $\frac{A_{s,{p_1}}}{a^{n+1}}\leq 1$,  the constant $\Lambda_3$ is uniformly bounded from above. Therefore \begin{equation}\label{Asbetaestimate}
    \begin{aligned}
        A_{s,{p_1}}=&\left(\int_{M}(\tilde{u}-s)_+^{p_1} f^{p_1}\omega_M^n\right)^{\frac{1}{{p_1}}}\\
        \leq &C_nA_{s,j,{p_1}}^{\frac{1}{n+1}}\left(\int_{\Omega_s}(-\psi_3+\Lambda_3)^\frac{n{p_1}}{n+1}f^{{p_1}}\omega_M^n\right)^{\frac{1}{{p_1}}}\quad \text{ (for $j$ large enough)}\\
        \leq& CA_{s,j,{p_1}}^{\frac{1}{n+1}}\left(\int_M(-\psi_3+1)^{\frac{n{p_1} p_2}{n+1}}f^{({p_1}-1)p_2+1}\omega_M^n\right)^{\frac{1}{{p_1} p_2}}\left(\int_{\Omega_s}f\omega_M^n\right)^{(1-\frac{1}{p_2})\frac{1}{p_1}}
    \end{aligned}
\end{equation}
    We note that form our construction of $f$,  we have $||f||_{L^q}\leq Ce^{\Lambda_1}$ for some uniform constant $C$. Note that $\psi_3$ is uniformly bounded, one can estimate the term \begin{equation}\begin{aligned}
             \left(\int_M(-\psi_3+1)^{\frac{n{p_1} p_2}{n+1}}f^{({p_1}-1)p_2+1}\omega_M^n\right)^{\frac{1}{{p_1} p_2}}\leq&C\left(\int_Mf^{(p_1-1)p_2+1}\omega_M^n\right)^{\frac{1}{{p_1} p_2}}
             \leq Ce^{\frac{\Lambda_1((p_1-1)p_2+1)}{{p_1} p_2}}
    \end{aligned}
    \end{equation}
    Hence equation \eqref{Asbetaestimate} becomes \begin{equation}
        A_{s,{p_1}}\leq CA_{s,j,{p_1}}^{\frac{1}{n+1}}e^{\frac{\Lambda_1((p_1-1)p_2+1)}{{p_1} p_2}}\phi(s)^{(1-\frac{1}{p_2})\frac{1}{p_1}}.
    \end{equation}
    Let $j\rightarrow \infty$, we get $A_{s,{p_1}}\leq Ce^{\frac{\Lambda_1((p_1-1)p_2+1)(n+1)}{{p_1} p_2n}}\phi(s)^{(1-\frac{1}{p_2})\frac{n+1}{np_1}}.$
    Now we specify the choice of $ p_1, p_2$. First we fix  $p_1$ close to $1$, and $p_2>n+1$, then there exists a constant $\delta_n>0$ depending only on $n$ such that $(1-\frac{1}{p_2})\frac{n+1}{np_1}=1+\delta_n$. We remark that $\delta_n$ can be chosen close to $\frac{1}{n}$ but strictly less than $\frac{1}{n}$ by choosing $p_2, p_2$ properly.  The lemma follows with $A_1={\frac{((p_1-1)p_2+1)(n+1)}{{p_1} p_2n}}$. 
\end{proof}

Next, we discuss the case when $\frac{A_{s,{p_1}}}{a^{n+1}}\geq 1 $. Since $A_{s,{p_1}}$ is decreasing in $s$, let $s_1$ be the largest positive number such that $A_{s,{p_1}}\geq a^{n+1}.$ Without loss of generality, we can assume that $A_{s_1,{p_1}}=a^{n+1}$. Then we have \begin{lemma}
    $s_1\leq \max\{s_0,Ce^{A_2\Lambda_1}\frac{||(v-u)_+||_{L^1}}{a^{(n+1){p_1} p_4}}\},$ where $s_0=2\sup(a\psi)+4a||v||_{L^\infty}$, $C$ is an uniform constant, $A_2$ is a dimensional constant, and $p_4$ is a constant with $\frac{1}{p_4}+\frac{{p_1}}{q}=1.$
\end{lemma}
\begin{proof}
    We need only to show that when $s_1\geq s_0$, there exists uniform constants $C,A_2$ such that $s_1\leq Ce^{A_2\Lambda_1}\frac{||(v-u)_+||_{L^1}}{a^{(n+1){p_1} p_4}}.$
On $\Omega_{s_1}$, we have $-u+(1-2a)v+a\psi\geq s_0$, so $-u+v\geq \frac{1}{2}s_0$ from our choice of $s_0=2\sup a\psi+4a||v||_{L^\infty}$. Then one can estimates $s_1$ from above as follows, choose $p_4>1$ with $\frac{1}{p_4}+\frac{{p_1}}{q}=1,$
 \begin{equation}
        \begin{aligned}
            a^{n+1}=A_{s_1,{p_1}}=&\left(\int_{\Omega_{s_1}}(-u+(1-2a)v+a\psi-s_1)_+^{p_1} f^{{p_1}}\omega_M^n\right)^{\frac{1}{{p_1}}}
            \leq C\left(\int_{\Omega_{s_1}}f^{{p_1}}\omega_M^n\right)^{\frac{1}{{p_1}}}\\
            \leq&C\left(\int_{\Omega_{s_1}}\frac{(v-u)_+^{\frac{1}{p_4}}}{s_1^{\frac{1}{p_4}}}f^{{p_1}}\omega_M^n\right)^{\frac{1}{{p_1}}}\leq Cs_1^{-\frac{1}{p_4{p_1}}}||(v-u)_+||_{L^1}^{\frac{1}{p_4{p_1}}}e^{\Lambda_1}.
        \end{aligned}
    \end{equation}
    The lemma follows with $A_2={p_1} p_4.$
\end{proof}
The next step is to estimate $\phi(s_0)$.\begin{lemma}
    There exists uniform constants $C, A_3$ such that $\phi(s_0)\leq Ce^{A_3\Lambda_1}||(v-u)_+||_{L^1}^\frac{1}{q^*},$ where $\frac{1}{q^*}+\frac{1}{q}=1.$ 
\end{lemma}
\begin{proof}Note that on $\Omega_{s_0}$, we have $-u+v\geq \frac{1}{2}s_0$ with a similar argument as above. So 
    \begin{equation}
        \begin{aligned}
          \phi(s_0)=\int_{\Omega_{s_0}}f\omega_M^n\leq\int_M \frac{(-u+v)_+^{\frac{1}{q^*}}}{s_0^{\frac{1}{q^*}}}f\omega_M^n\leq s_0^{-\frac{1}{q^*}}e^{\Lambda_1}||(v-u)_+||_{L^1}^\frac{1}{q^*} . 
        \end{aligned}
    \end{equation}
    Note that $s_0\geq 2\sup (a\psi_2)=-rh(\Lambda_1)\geq C\int_{\Lambda_1}^{C_0+\Lambda_1}\Phi^{-\frac{1}{n}}(x)dx\geq CC_0(\Lambda_1+C_0)^{-\alpha}$. So the lemma follows.
\end{proof}

Now we can close the proof of proposition \ref{stabilityestimate} by choosing $\Lambda_1$ appropriately. We need the following De Giorgi iteration lemma \begin{lemma}[De Giorgi iteration]\label{DeGiorgi}Let $\phi(s):[0,\infty)\rightarrow [0,\infty)$ be a decreasing and continuous function which satisfies $t\phi(s+t)\leq A_0\phi(s)^{1+\delta}$ for some constant $A_0>0,\delta>0$ and all $t>0,s\geq s_0>0$. Then $\phi(s)=0$ for \[
s\geq s_\infty:=s_0+\frac{2A_0\phi(s_0)^{\delta}}{1-2^{-\delta}}.
\]
\end{lemma}
From De Giorgi iteration lemma, we get \begin{equation}
    \begin{aligned}\tilde u&\leq s_1+Ce^{A_1\Lambda_1}\phi(s_1)^{\delta_n}\\
    &\leq \max\{s_0,Ce^{A_2\Lambda_1}\frac{||(v-u)_+||_{L^1}}{a^{(n+1){p_1} p_4}}\}+Ce^{A_1\Lambda_1}\phi(s_0)^{\delta_n}.\end{aligned}
\end{equation}
Choose $\Lambda_1=\frac{1}{B}\log(||(v-u)_+||_{L^1})$ for some uniform large $B$. Then \[
\sup(a\psi_2)=\frac{q}{\alpha_M}e^{\frac{c-c_1}{n}}\int_{\Lambda_1}^{C_0+\Lambda_1}x^{-\alpha}dx\leq C\Lambda_1^{-\alpha}.
\] 
The relation between $a$ and $\Lambda_1$ is $\Phi^\frac{1}{n}(\Lambda_1)=\frac{1}{a}e^{\frac{c-c_1}{n}}$.  So  we also have $a\leq C\Lambda_1^{-\alpha}$. Therefore, we get the desired estimate \begin{equation}
    \sup(v-u)\leq C\Lambda_1^{-\alpha}+Ce^{C\Lambda_1}\leq\frac{C}{ |\log(||(v-u)_+||_{L^1})|^{\alpha}}.
\end{equation}

\begin{remark}\begin{enumerate}
    \item As this note was being written, we note that Fang also proved similar stability results in \cite{Fang25}. However, our approach is completely different, where her argument is based on the pluripotential theory, while here we only use the PDE techniques. 
    \item It's not hard to extend the above stability estimate to a more general class of complex hessian type equations satisfying the determinant domination condition introduced by Guo-Phong-Tong \cite{GP23}. We leave it for interested readers. The difficulty of using such a stability estimate to get the modulus of continuity comes from that we don't have a good approximation of $m$-subharmonic functions. See \cite{CX22} for the recent progress. 
    \item We only use the condition $\alpha<\frac{p}{n}-1$ in order to construct a comparison Hermitian metric $\omega_M+\ddc\psi_1$ which satisfies the condition \begin{equation}
        (\omega_M+\ddc\psi_1)\geq \delta_0\Phi(F)e^{F}\omega_M^n
 \text{ for some uniform positive constant }\delta_0,   \end{equation} and for some positive increasing function $\Phi$ which satisfies $\int_0^\infty\Phi^{-\frac{1}{n}}(t)dt<\infty$.

 Another method to construct such $\omega_M+\ddc\psi_1$ is by solving the Monge-Amp\`ere equation directly, but one needs to control a lower bound of the constant $e^c$, (defined by Guo-Song \cite{GS24}, the sup-slope of the equations) appearing in solving the equation. It would be interesting to get more information about it. The method we applied above is to construct such Hermitian metric locally and then glue it to a global one. We are only able to glue it when the locally constructed metric potential is uniformly bounded. 

 Another concept used by Geudj-Lu in \cite{GL2} to get a lower bound of $e^c$ is \begin{equation}
     v_-(\omega):=\inf\{\int_{M}(\omega+\ddc\varphi)^n|\varphi\in Psh(M,\omega)\cap L^{\infty}\}.
 \end{equation}
 It's clear that $e^c\geq v_-(\omega_M)$ for any smooth $F$ with $\int_Me^F\omega_M^n=1$. In some cases, one has a positive lower bound of $v_-(\omega_M)$, for instance, when the manifold admits a nef and big closed real $(1,1)$-form $\beta$, see Lemma 9 in \cite{GP23}.
\end{enumerate}
\end{remark}

\section{log-continuity}
Now we introduce the approximation of plurisubharmonic functions we need. Let $\rho:\mathbb{C}^n\rightarrow \mathbb{R}$ be a smooth kernel such that $\int_{\mathbb{C}^n}\rho(|w|)dV(w)=1$ and $\rho(t)=1$ when $t\leq \frac{3}{4}$ and has compact support in the unit ball. One defines \[
    \rho_\delta u(x):=\frac{1}{\delta^{2n}}\int_{T_xM}u(\exph_x(\zeta))\rho(\frac{|\zeta|_{\omega_x}^2}{\delta^2})dV_{\omega_x}(\zeta),
\]where $\exph$ is the holomorphic part of the exponential map with respect to the Chern connection associated with Hermitian metric $\omega$. For details about this holomorphic part of the exponential map, see \cite{MR1319346} proposition 2.9. 

We list some properties of $\rho_\delta u$ we need in the proof as follows:
\begin{proposition}[Berman-Demailly\cite{BD12}, lemma 1.2 in  see also \cite{GG23} ]\label{l1approximation}Let $u\in psh(M,\omega)$ with normalization $1\leq u\leq C$. We have \begin{enumerate}
\item There exists a constant $K>0$ depends on the curvature of $\omega$, such that $\rho_tu+Kt$ is monotone increasing in $t$. 
\item For any $c_1>0$, define $U_{c_1,\delta}=\inf_{t\in(0,\delta)}(\rho_tu(x)+Kt-c_1\log(\frac{t}{\delta})-K\delta)$, then \[
    \omega+\ddc U_{c_1,\delta}\geq -(Ac_1+K\delta)\omega,
\]where $-A$ is the lower bound of the bidectional curvature of $\omega$.
\item $\int_M|\rho_\delta u-u|\omega^n\leq C(n,\omega)\delta^2, \quad \forall \delta\in(0,1].$
\end{enumerate}
\end{proposition}
With this approximation lemma, one can prove that 
\begin{lemma}There exists a uniform constant $\delta_1>0$ such that for all $\delta<\delta_1$, one has \begin{equation}
    |\rho_\delta u(x)-u(x)|\leq \frac{C}{|\log\delta|^{\alpha}},
\end{equation}
    \end{lemma}
\begin{proof}One side of the inequality $\rho_\delta u-u\geq -K\delta$ follows from proposition \ref{l1approximation} that $\rho_t u+Kt$ is increasing in $t$. We prove the other side as follows:

 We choose a small $c_1=|\log\delta|^{-\alpha}$, then $c_1$ is large compared with $\delta$. So $\omega_M+\ddc U_{c_1,\delta}\geq -A'c_1\omega_M$. Hence we can define a $\omega_M$-psh function $u_{c_1,\delta}:=(1-A'c_1)U_{c_1,\delta}$. It follows that $\omega_M+\ddc u_{c_1,\delta}\geq (A'c_1)^2\omega_M\geq 0.$ And from the normalization of $u$ we have $u_{c_1,\delta}>0$.  From our definition of $u_{c_1,\delta}$, we have $u_{c_1,\delta}\leq U_{c_1,\delta}\leq \rho_\delta u$, and hence \begin{equation}
     ||(u_{c_1,\delta}-u)_+||_{L^1}\leq ||(\rho_\delta u-u)_+||_{L^1}\leq C\delta^2.
 \end{equation} 
 By proposition \ref{stabilityestimate}, we have \[
 u_{c_1,\delta}(x)-u(x)\leq \frac{C}{|\log\delta|^\alpha}.
 \]
 From the definition of $u_{c_1,\delta}=(1-Ac_1)U_{c_1,\delta}$ and $$U_{c_1,\delta}=\inf_{t\in(0,\delta]}\{\rho_t u+Kt-c_1\log(\frac{t}{\delta})-K\delta\},$$
one gets there exists $t_x\in (0,\delta]$ such that \[
    \rho_{t_x}u(x)+Kt_x-c_1\log(\frac{t_x}{\delta})-K\delta=U_{c_1,\delta}(x)\leq u(x)+A'c_1U_{c_1,\delta}(x)+2ru_{c_1,\delta}+\frac{C}{|\log \delta|^\alpha}.
\]
Since $\rho_tu+Kt$ is increasing in $t$, we have $\rho_{t_x}u(x)+Kt_x\geq u(x)$. 
So \[
    -\log(\frac{t_x}{\delta})\geq K\frac{\delta}{c}+A'U_{c_1,\delta}+C.
\]
From the choice of $c_1$ we have \[
    -\log\frac{t_x}{\delta}\geq -C.
\] In other words, there exists a uniform constant $\theta$ such that $t_x\geq \theta\delta$. So using the fact that $\rho_tu+KT$ is increasing in $t$ again, and the observation $\delta\leq c_1$, we have \[
\begin{aligned}
  \rho_{\theta\delta}+K\theta t-u\leq &\rho_{t_x}+Kt_x-u\\
  \leq & c_1\log(\frac{t_x}{\delta})+K\delta+A'c_1U_{c_1,\delta}+\frac{C}{|\log\delta|^\alpha}\\
  \leq &\frac{C}{|\log\delta|^\alpha}.
\end{aligned}
\]
It implies that there is a universal constant $\delta_1>0$ such that $\forall \delta\leq  \delta_1$, we have \[
    \rho_\delta u(x)-u(x)\leq \frac{C}{|\log\delta|^\alpha}.
\]
\end{proof}

Next, one can improve the integral estimate above to the modulus of continuity in a standard way as in \cite{MR3191972, MR4398254}. For the reader's convenience, we include it here.  \begin{lemma}Suppose $|\rho_\delta u(x)-u(x)|\leq C|\log\delta|^{-\alpha}$ for all $x\in M$, and for all $0<\delta\leq \delta_1$, then there exists uniform constants $\delta_2,C>0$ such that 
    \begin{equation}
        |u(x)-u(y)|\leq \frac{C}{|\log\dist_{\omega_M}(x,y)|^\alpha},\quad \forall \delta\leq \delta_2.
    \end{equation}
\end{lemma}
\begin{proof}
    Set $\omega(\delta):=\sup_{\dist_{\omega_M}(x,y)\leq \delta}(u(x)-u(y)).$  We argue by contradiction. 
    If the lemma is false, there exists a sequence of points $x_i,y_i\in M$ and $\delta_i\rightarrow 0$ such that \begin{equation}
        \dist_{\omega_M}(x_i,y_i)\leq\delta_i,\quad \omega(\delta_i)=u(x_i)-u(y_i), \quad \omega(\delta_i)|\log\delta_i|^{\alpha}\rightarrow \infty.
    \end{equation}
    Since $M$ is compact, we can assume $x_i,y_i\rightarrow x_0$ as $i\rightarrow \infty$. Let $B$ be a small neighborhood of $x_0$ which can identified as a unit ball in $\CC^n$, and in $B$ one has $(1-\epsilon_0)\dist_{\CC^n}(x,y)\leq \dist_{\omega_M}(x,y)\leq (1+\epsilon_0) \dist_{g_{\CC^n}}(x,y)$. By adding a smooth function, one can assume that $u$ is a plurisubharmonic function in $B$. Fix $b>1$, which will be specified later, 
    \begin{equation}
        \begin{aligned}
            \dashint_{B(x_{i},(b+1)a)}u(x)dV(x)=&\frac{b^{2n}}{(b+1)^{2n}}\dashint_{B(y_{i},b\delta_i)}u(x)dV(x)\\
            &+\frac{1}{|B(x_{i},b+1)\delta_i)}\int_{B(x_i,(b+1)\delta_i)\setminus B(y_i,b\delta_i)}u(x)dV(x)\\
            \geq& \frac{b^{2n}}{(b+1)^{2n}}u(y_i)+(1-\frac{b^{2n}}{(b+1)^{2n}})(u(y_i)-\omega((1+\epsilon_0)(b+2)\delta_i))\\
            =&u(y_i)-(1-\frac{b^{2n}}{(b+1)^{2n}})\omega((1+\epsilon_0)(b+2)\delta_i).
        \end{aligned}
    \end{equation}
    Here we have used the mean value inequality for $u$ and the fact $dist_{\omega_M}(x,y_i)\leq(1+\epsilon_0)(b+2)\delta_i $ for $x\in B(x_i,(b+1)\delta_i)\setminus B(y_i,b\delta_i),$ in the second line. 
    On the other hand \begin{equation}\label{pointwiseestimate}
        \begin{aligned}
            &\dashint_{B(x_i,2(b+1)\delta_i)}u(x)\rho(\frac{|x-x_i|^2}{(2(b+1)\delta_i)^2})dV(x)-u(x_i)\\
            =&|\partial B(0,1)|\int_0^1(\dashint_{\partial B(x_i,2(b+1)t)}ud\sigma(x)-u(x_i))t^{2n-1}dt\\
            \geq &|\partial B(0,1)|\int_{\frac{1}{2}}^1(\dashint_{B(x_{i},(b+1)a)}u(x)dV(x)-u(x_i))t^{2n-1}dt\\
            \geq &C(u(y_i)-u(x_i)-(1-\frac{b^{2n}}{(b+1)^{2n}})\omega((1+\epsilon_0)(b+2)\delta_i))\\
            = &C(\omega(\delta_i)-(1-\frac{b^{2n}}{(b+1)^{2n}})\omega((1+\epsilon_0)(b+2)\delta_i)).
        \end{aligned}
    \end{equation}
    From the expansion of $\exph$(Remark 4.6 in \cite{MR1319346}), we  have \begin{equation}
        \rho_\delta u(x)=\dashint_{B(x_i,\delta)}u(x)\rho(\frac{|x-x_i|^2}{\delta_i^2})dV(x)+O(\delta^2).
    \end{equation} 
So the assumption in the lemma implies \begin{equation}\label{integralestimate}
    \dashint_{B(x_i,2(b+1)\delta_i)}u(x)\rho(\frac{|x-x_i|^2}{(2(b+1)\delta_i)^2})dV(x)-u(x)\leq C|\log(2(b+1)\delta_i)|^{-\alpha}.
\end{equation}
Combine \eqref{pointwiseestimate} and \eqref{integralestimate}, we get \begin{equation}
    \omega(\delta_i)-(1-\frac{b^{2n}}{(b+1)^{2n}})\omega((1+\epsilon_0)(b+2)\delta_i))\leq C|\log(2(b+1)\delta_i)|^{-\alpha}.
\end{equation}
    Set $\tau(\delta):=\omega(\delta)|\log(\delta)|^{\alpha}$. Then \begin{equation}
        \tau(\delta_i)\leq (1-\frac{b^{2n}}{(b+1)^{2n}})\frac{|\log(\delta_i)|^\alpha}{|\log((1+\epsilon_0)(b+2)\delta_i)|^\alpha}\tau((1+\epsilon_0)(b+2)\delta_i))+C
    \end{equation}
    Choose $b$ large such that $(1-\frac{b^{2n}}{(b+1)^{2n}})\leq \frac{1}{2}$ and choose $\delta_1$ small enough, then $$(1-\frac{b^{2n}}{(b+1)^{2n}})\frac{|\log(\delta_i)|^\alpha}{|\log((1+\epsilon_0)(b+2)\delta_i}<\frac{3}{4}.$$ Then \begin{equation}
        \tau(\delta_i)\leq \frac{3}{4}\tau(\gamma\delta_i)+C,\quad \gamma=(1+\epsilon_0)(b+2).
    \end{equation}
    Iterate this inequality, we get a contradiction with $\tau(\delta_i)\rightarrow \infty$. 
\end{proof}

\begin{remark}
    We remark that when one assumes a uniform bound of $\int_M\log^p(|F|+1)(|F|+1)^ne^F\omega_M^n$, for some $p>n$, then with a similar argument, one can prove a uniform modulus of continuity of $\log\log$-type:\[
    |u(x)-u(y)|\leq \frac{C}{(\log|\log(\dist_{\omega_M}(x,y))|)^\alpha}.
    \]
\end{remark}
A corollary of the $\log$-continuity is the diameter bound of the metric. The proof of lemma 5 in \cite{GPTW} shows:
\begin{corollary}
    Fix $p>3n$, let $u$ be a smooth solution to equation \eqref{maequation}, then the diameter of metric $\omega_M+\ddc u$ is bounded uniformly in terms of $M,\omega_M,n,p$, an upper bound of $\int_M(|F|^p+1)e^F$ and a positive lower bound of $\int_Me^{\frac{1}{n}F}\omega_M^n.$
\end{corollary}

\section{Extending to bounded solutions}In this section, we want to prove that any bounded solution to the equation \eqref{maequation} in the sense of Bedford-Taylor also satisfies the log-continuity stated in the main theorem. We follow the trick of Lu-Phung-Tô \cite{MR4398254} to reduce it to the following stability estimate:\begin{proposition}\label{stabilityappro}
    Fix $e^F,e^G\in C^\infty(M,\mathbb{R})$. Assume $u,v$ are $\omega$-psh functions on $M$ satisfying \begin{equation*}
        (\omega_M+\ddc u)^n=e^{u+F}\omega_M^n \quad \text{and }\quad (\omega_M+\ddc v)^n=e^{v+G}\omega_M^n.
    \end{equation*} Then there exists a constant $C$ depending on $M,\omega_M$, $n,p$, an upper bounds for $\int_M(|F|^p+1)e^F\omega_M^n$, $\int_M(|G|^p+1)e^G\omega_M^n$, and a positive lower bound for $\int_Me^{\frac{F}{n}}\omega_M^n$, $\int_Me^{\frac{G}{n}}\omega_M^n$, such that \begin{equation*}
        ||u-v||_{L^\infty}\leq C||e^F-e^G||_{\Phi}^{\frac{1}{n}}
    \end{equation*} 
\end{proposition}
Here, we will use the notation that \[
||e^F||_{\Phi}:=\inf\{t>0|\int_{M}\Phi(F-\log t)e^{F-\log t}\omega_M^n\leq 1\}.
\]The norm $||e^F||_\Phi$ is equivalent to the Orlicz norm of $e^F$ with the weight grows like $x\Phi(\log x).$ For simplicity, we fix a increasing function $\Phi(x):\mathbb{R}\to \mathbb{R}^+$ with $\Phi(x)=x^p$ for some $p>n$ when $x>1$.
Such form of stability estimate is easier in the K\"ahler case for equation $(\omega_M+\ddc u)^n=e^{F}\omega_M^n$ because the comparison principle holds. On Hermitian manifolds, a modified comparison principle was proved by Kolodziej-Cuong and they used it to prove a stability estimate with an additional non-degeneracy assumption that $e^F>c_0>0$. But for the equation $(\omega_M+\ddc u)^n=e^{u+F}\omega_M^n$, one can remove the nondegeneracy conditions on $e^F$(\cite{MR4398254}).  
\begin{proof}
First, we note that the solution $u,v$ are uniformly bounded: We consider the following equation: \[
(\omega_M+\ddc\psi)^n=e^{F+c_F}\omega_M^n, \omega_M+\ddc\psi>0, \sup \psi=0.
\]By lemma \ref{estimatesupslope}, we know that $e^{c_F}+e^{-c_F}+||\varphi||_{L^\infty}\leq C_0$ for some uniform constant $C_0$. Then we can apply the maximum principle to get the uniform estimate for $u$ as follows: let $x_0$ be the point where $u-\varphi$ attains its maximum, then at $x_0$, we have \[
0\geq \ddc(u-\psi)=\omega_u-\omega_\psi.
\]Combining with the equation we get $
e^{u+F}\leq  e^{c_F+F}$, which implies $u(x_0)\leq e^{c_F}$. It follows that \[
\max u\leq \max(u-\psi)+\max\psi\leq u(x_0)-\min\psi\leq 2C_0.
\]Similary one can prove that $u\geq -2C_0$. So let $C_1$ be a uniform upper bound for $||u||_{L^\infty},||v||_{L^\infty}$. 

    We consider the following auxiliary equation \[
    (\omega_M+\ddc\varphi)^n=e^{c_{F,G}}\left(\frac{|e^F-e^G|}{||e^F-e^G||_\Phi}+1\right)\omega_M^n
    \]
    Then it is clear that the $L^1\Phi(\log L)$-integral of $\frac{|e^F-e^G|}{||e^F-e^F||_\Phi}$ is bounded by $1$ from our definition of $||\cdot||_\Phi$. Therefore from lemma \ref{estimatesupslope}, we have $e^{c_{F,G}}+e^{-c_{F,G}}+||\varphi||_{L^\infty}$ is uniformly bounded, say by $C_2$.  

    Set $\epsilon=e^{(C_1-C_2)/n}||e^F-e^G||_{\Phi}.$ If $\epsilon\leq \frac{1}{2},$ we'd like to apply maximum principle to the function \[
    w=(1-\epsilon)u+\epsilon\varphi-C_1\epsilon+n\log(1-\epsilon),
    \]
    By our choice of $\epsilon$, we have $\frac{\epsilon^ne^{c_{F,G}}}{||e^F-e^G||_\Phi}=e^{C_1+c_{F,G}-C_2}\geq e^{u+n\log(1-\epsilon)}$, and for $w$ we also have $w\leq u+n\log(1-\epsilon)$. Therefore we can compute \[
    \omega_w^n\geq (1-\epsilon)^n\omega_u^n+\epsilon^n\omega_\varphi^n\geq e^{u+n\log(1-\epsilon)}e^G\omega_M^n\geq e^{w}\omega_M^n.
    \]
    Apply the maximum principle to $w-v$: at the maximum point of $w-v$, say $x_0$, we have $\omega_w\leq \omega_v$, which imples \[
    e^{w+G}\leq e^{v+G},
    \] so $w\leq v.$ From the expression of $w$, we get \[
    u-v\leq C_3\epsilon 
    \] for a uniform constant $C_3$. Similarly, one can prove the reverse $v-u\leq C_3\epsilon$. 
    In summary, we get in the case of $\epsilon=e^{(C_1-C_2)/n}||e^F-e^G||_\Phi\leq \frac{1}{2}$, we have $|u-v|_{L^\infty}\leq C_4||e^F-e^G||_{\Phi}^\frac{1}{n}$, for a uniform constant $C_4$.
    When, $\epsilon\geq \frac{1}{2}$, it's much easier since we have a uniform lower positive lower bound of $||e^{F}-e^{G}||_\Phi$, and enlarge the constant $C_4$ if necessary, we still get \[
    |u-v|_{L^\infty}\leq C_4||e^F-e^G||_{\Phi}^\frac{1}{n}
    \]  
\end{proof}
With such a stability lemma, one can pass the previous log-continuity in theorem \ref{mainthm} to the bounded solution $u$. Assume the integral bounds for $F$ as in theorem \ref{mainthm}, and let $u$ be a bounded solution (in the Bedford-Taylor sense) to equation \[
\omega_u^n=e^{F+c}\omega_M^n, \omega_u\geq0,\sup u=0.
\]Let $G_j$ be a sequence of smooth function such that $G_j\leq G_{j+1}\to -u+F+c$ pointwisely. Then we have \[
\int_M\Phi(G_j)e^{G_j}\omega_M^n\leq \int_M\Phi(-u+F+c)e^{-u+F+c}\omega_M^n\leq C,\]for a uniform constant $C$. And by dominant convergence theorem, $\int_Me^{G_j/n}\omega^n\to \int_{M}e^{-u+F+c}\omega_M^n$. So we have a uniform positive lower bound of $\int_{M}e^{G_j/n}\omega^n$. Moreover, we claim that $||e^{G_j}-e^{-u+F+c}||_\Phi\rightarrow 0$. For simplicity, we let $G:=-u+F+c$.  Assume this is not true, then there exist $\delta>0$ such that \[
\int_M\Phi(\log(e^G-e^{G_j})-\log \delta)(e^G-e^{G_j})/\delta \geq 1.
\] 
But by the dominant convergence theorem, the above integral converges to $0$, we get a contradiction. 

Now consider the following family of equations\[
\omega_{u_j}^n=e^{u_j+G_j}\omega_M^n,\omega_{u_j}>0.
\] The by theorem \ref{mainthm}, the solution $u_j$ has the stated uniform log-continuity. And by passing to a subsequence, with proposition \ref{stabilityappro}, $u_j\to u_\infty\in L^\infty$. Note that $u_\infty$ solves \[
\omega_{u_\infty}^n=e^{u_\infty-u+F+c}\omega_M^n
\]Note that $u$ is a solution to the limit equation, and by uniqueness(theorem 0.1 in \cite{Ng16}) we get $u_\infty=u$. It follows that $u$ satisfies the stated log-continuity.
\AtNextBibliography{\small}
\begingroup
\setlength\bibitemsep{2pt}
\printbibliography

@incollection {MR1319346,
    AUTHOR = {Demailly, Jean-Pierre},
     TITLE = {Regularization of closed positive currents of type {$(1,1)$}
              by the flow of a {C}hern connection},
 BOOKTITLE = {Contributions to complex analysis and analytic geometry},
    SERIES = {Aspects Math.},
    VOLUME = {E26},
     PAGES = {105--126},
 PUBLISHER = {Friedr. Vieweg, Braunschweig},
      YEAR = {1994},
      ISBN = {3-528-06633-4},
   MRCLASS = {32C30 (32J25)},
  MRNUMBER = {1319346},
MRREVIEWER = {Mongi\ Blel},
}

@article {MR3191972,
    AUTHOR = {Demailly, Jean-Pierre and Dinew, S\l awomir and Guedj, Vincent
              and Pham, Hoang Hiep and Ko\l odziej, S\l awomir and Zeriahi,
              Ahmed},
     TITLE = {H\"older continuous solutions to {M}onge-{A}mp\`ere equations},
   JOURNAL = {J. Eur. Math. Soc. (JEMS)},
  FJOURNAL = {Journal of the European Mathematical Society (JEMS)},
    VOLUME = {16},
      YEAR = {2014},
    NUMBER = {4},
     PAGES = {619--647},
      ISSN = {1435-9855,1435-9863},
   MRCLASS = {32W20 (32Q15 32U05 32U15 32U40 35B65 35J96 53C55)},
  MRNUMBER = {3191972},
MRREVIEWER = {Muhammed\ Ali\ Alan},
       DOI = {10.4171/JEMS/442},
       URL = {https://doi.org/10.4171/JEMS/442},
}

@article {MR4398254,
    AUTHOR = {Lu, Chinh H. and Phung, Trong-Thuc and T\^o, T\^at-Dat},
     TITLE = {Stability and {H}\"older regularity of solutions to complex
              {M}onge-{A}mp\`ere equations on compact {H}ermitian manifolds},
   JOURNAL = {Ann. Inst. Fourier (Grenoble)},
  FJOURNAL = {Universit\'e{} de Grenoble. Annales de l'Institut Fourier},
    VOLUME = {71},
      YEAR = {2021},
    NUMBER = {5},
     PAGES = {2019--2045},
      ISSN = {0373-0956,1777-5310},
   MRCLASS = {32W20 (32Q15 32U05)},
  MRNUMBER = {4398254},
MRREVIEWER = {Rafa\l\ Czy\.z},
       DOI = {10.5802/aif.3436},
       URL = {https://doi.org/10.5802/aif.3436},
}

@misc{liu2024complexalexandrovbakelmanpucciestimateapplications,
      title={Complex Alexandrov-Bakelman-Pucci estimate and its applications}, 
      author={Junbang Liu},
      year={2024},
      eprint={2410.04395},
      archivePrefix={arXiv},
      primaryClass={math.DG},
      url={https://arxiv.org/abs/2410.04395}, 
}

@misc{liu24c,
title={A new PDE method to sharp uniform estimates for complex Monge-Ampere type equations},
Author={Junbang Liu},
url={https://sites.google.com/stonybrook.edu/liu/research},
}

@article {TW10,
    AUTHOR = {Tosatti, Valentino and Weinkove, Ben},
     TITLE = {The complex {M}onge-{A}mp\`ere equation on compact {H}ermitian
              manifolds},
   JOURNAL = {J. Amer. Math. Soc.},
  FJOURNAL = {Journal of the American Mathematical Society},
    VOLUME = {23},
      YEAR = {2010},
    NUMBER = {4},
     PAGES = {1187--1195},
      ISSN = {0894-0347,1088-6834},
   MRCLASS = {32W20 (32Q15)},
  MRNUMBER = {2669712},
MRREVIEWER = {S\l awomir\ Ko\l odziej},
       DOI = {10.1090/S0894-0347-2010-00673-X},
       URL = {https://doi.org/10.1090/S0894-0347-2010-00673-X},
}

@article {KN15,
    AUTHOR = {Ko\l odziej, S\l awomir and Nguyen, Ngoc Cuong},
     TITLE = {Weak solutions of complex {H}essian equations on compact
              {H}ermitian manifolds},
   JOURNAL = {Compos. Math.},
  FJOURNAL = {Compositio Mathematica},
    VOLUME = {152},
      YEAR = {2016},
    NUMBER = {11},
     PAGES = {2221--2248},
      ISSN = {0010-437X,1570-5846},
   MRCLASS = {32U40 (35D30 35J96 53C55)},
  MRNUMBER = {3577893},
MRREVIEWER = {Dongrui\ Wan},
       DOI = {10.1112/S0010437X16007417},
       URL = {https://doi.org/10.1112/S0010437X16007417},
}

@article {CNS2,
    AUTHOR = {Caffarelli, L. and Kohn, J. J. and Nirenberg, L. and Spruck,
              J.},
     TITLE = {The {D}irichlet problem for nonlinear second-order elliptic
              equations. {II}. {C}omplex {M}onge-{A}mp\`ere, and uniformly
              elliptic, equations},
   JOURNAL = {Comm. Pure Appl. Math.},
  FJOURNAL = {Communications on Pure and Applied Mathematics},
    VOLUME = {38},
      YEAR = {1985},
    NUMBER = {2},
     PAGES = {209--252},
      ISSN = {0010-3640,1097-0312},
   MRCLASS = {35J65 (58G30)},
  MRNUMBER = {780073},
MRREVIEWER = {Philippe\ Delano\"e},
       DOI = {10.1002/cpa.3160380206},
       URL = {https://doi.org/10.1002/cpa.3160380206},
}

@misc{DemBook,
Author={Jean-Pierre Demailly},
Title={Complex Analytic and
Differential Geometry},
url={https://www-fourier.ujf-grenoble.fr/~demailly/manuscripts/agbook.pdf}}

@misc{CX24,
      title={Viscosity solution to complex Hessian equations on compact Hermitian manifolds}, 
      author={Jingrui Cheng and Yulun Xu},
      year={2024},
      eprint={2406.00953},
      archivePrefix={arXiv},
      primaryClass={math.AP},
      url={https://arxiv.org/abs/2406.00953}, 
}

@article {GL2,
    AUTHOR = {Guedj, Vincent and Lu, Chinh H.},
     TITLE = {Quasi-plurisubharmonic envelopes 2: {B}ounds on
              {M}onge-{A}mp\`ere volumes},
   JOURNAL = {Algebr. Geom.},
  FJOURNAL = {Algebraic Geometry},
    VOLUME = {9},
      YEAR = {2022},
    NUMBER = {6},
     PAGES = {688--713},
      ISSN = {2313-1691,2214-2584},
   MRCLASS = {32W20 (32J18 32U05 35A23 35J96)},
  MRNUMBER = {4518244},
MRREVIEWER = {S\l awomir\ Dinew},
       DOI = {10.14231/ag-2022-021},
       URL = {https://doi.org/10.14231/ag-2022-021},
}

@misc{GS24,
      title={Sup-slopes and sub-solutions for fully nonlinear elliptic equations}, 
      author={Bin Guo and Jian Song},
      year={2024},
      eprint={2405.03074},
      archivePrefix={arXiv},
      primaryClass={math.AP},
      url={https://arxiv.org/abs/2405.03074}, 
}

@article{GP23,
author = {Bin Guo and Duong Phong},
title = {{On $L^\infty$ estimates for fully non-linear partial differential equations}},
volume = {200},
journal = {Annals of Mathematics},
number = {1},
publisher = {Department of Mathematics of Princeton University},
pages = {365 -- 398},
year = {2024},
doi = {10.4007/annals.2024.200.1.6},
URL = {https://doi.org/10.4007/annals.2024.200.1.6}
}

@incollection {BD12,
    AUTHOR = {Berman, Robert and Demailly, Jean-Pierre},
     TITLE = {Regularity of plurisubharmonic upper envelopes in big
              cohomology classes},
 BOOKTITLE = {Perspectives in analysis, geometry, and topology},
    SERIES = {Progr. Math.},
    VOLUME = {296},
     PAGES = {39--66},
 PUBLISHER = {Birkh\"auser/Springer, New York},
      YEAR = {2012},
      ISBN = {978-0-8176-8276-7},
   MRCLASS = {32U05 (32Q15 32U40 32W20)},
  MRNUMBER = {2884031},
MRREVIEWER = {Vincent\ Guedj},
       DOI = {10.1007/978-0-8176-8277-4\_3},
       URL = {https://doi.org/10.1007/978-0-8176-8277-4_3},
}

@misc{GG23,
      title={Diameter of K\"ahler currents}, 
      author={Vincent Guedj and Henri Guenancia and Ahmed Zeriahi},
      year={2023},
      eprint={2310.20482},
      archivePrefix={arXiv},
      primaryClass={math.DG},
      url={https://arxiv.org/abs/2310.20482}, 
}

@misc{GPTW,
      title={On the modulus of continuity of solutions to complex Monge-Amp\`ere equations}, 
      author={Bin Guo and Duong H. Phong and Freid Tong and Chuwen Wang},
      year={2021},
      eprint={2112.02354},
      archivePrefix={arXiv},
      primaryClass={math.DG},
      url={https://arxiv.org/abs/2112.02354}, 
}

@article {KN19,
    AUTHOR = {Ko\l odziej, S\l awomir and Nguyen, Ngoc Cuong},
     TITLE = {Stability and regularity of solutions of the
              {M}onge-{A}mp\`ere equation on {H}ermitian manifolds},
   JOURNAL = {Adv. Math.},
  FJOURNAL = {Advances in Mathematics},
    VOLUME = {346},
      YEAR = {2019},
     PAGES = {264--304},
      ISSN = {0001-8708,1090-2082},
   MRCLASS = {53C55 (32U40 32W20 35J96)},
  MRNUMBER = {3910489},
MRREVIEWER = {Masaya\ Kawamura},
       DOI = {10.1016/j.aim.2019.02.004},
       URL = {https://doi.org/10.1016/j.aim.2019.02.004},
}

@article {Ng16,
    AUTHOR = {Nguyen, Ngoc Cuong},
     TITLE = {The complex {M}onge-{A}mp\`ere type equation on compact
              {H}ermitian manifolds and applications},
   JOURNAL = {Adv. Math.},
  FJOURNAL = {Advances in Mathematics},
    VOLUME = {286},
      YEAR = {2016},
     PAGES = {240--285},
      ISSN = {0001-8708,1090-2082},
   MRCLASS = {32W20 (32U05 32U40 53C55)},
  MRNUMBER = {3415685},
MRREVIEWER = {Bianca\ Santoro},
       DOI = {10.1016/j.aim.2015.09.009},
       URL = {https://doi.org/10.1016/j.aim.2015.09.009},
}

@article{Di16,
     author = {Dinew, S{\l}awomir},
     title = {Pluripotential theory on compact {Hermitian} manifolds},
     journal = {Annales de la Facult\'e des sciences de Toulouse : Math\'ematiques},
     pages = {91--139},
     publisher = {Universit\'e Paul Sabatier, Toulouse},
     volume = {Ser. 6, 25},
     number = {1},
     year = {2016},
     doi = {10.5802/afst.1488},
     mrnumber = {3485292},
     zbl = {1342.32022},
     language = {en},
     url = {https://www.numdam.org/articles/10.5802/afst.1488/}
}

@misc{Fang25,
      title={Continuity of solutions to complex Hessian equations on compact Hermitian manifolds}, 
      author={Yuetong Fang},
      year={2025},
      eprint={2510.14690},
      archivePrefix={arXiv},
      primaryClass={math.CV},
      url={https://arxiv.org/abs/2510.14690}, 
}

@misc{CX22,
      title={Regularization Of $m$-subharmonic Functions And H\"Older Continuity}, 
      author={Jingrui Cheng and Yulun Xu},
      year={2022},
      eprint={2208.14539},
      archivePrefix={arXiv},
      primaryClass={math.AP},
      url={https://arxiv.org/abs/2208.14539}, 
}

\end{document}